\documentclass[12pt]{article}
\usepackage{e-jcnew}

\usepackage{amsmath,amssymb}
\usepackage{graphicx}
\usepackage[colorlinks=true,citecolor=black,linkcolor=black,urlcolor=blue]{hyperref}
\usepackage{enumerate}

\dateline{August 30, 2017}{June 8, 2018}{TBD} 

\MSC{05C80, 05C15}

\Copyright{  The authors. Released under the CC BY-ND license (International 4.0).}

\def\cX{{\cal X}}
\def\cZ{{\cal Z}}

\newcommand{\bfrac}[2]{\left(\frac{#1}{#2}\right)}

\def\cE{{\cal E}}

\newcommand{\beq}[1]{\begin{equation}\label{#1}}
\newcommand{\eeq}{\end{equation}}

\newcommand{\set}[1]{\left\{#1\right\}}

\DeclareMathOperator{\sign}{sgn}

\def\Pr{\mbox{{\bf Pr}}}
\def\whp{{\bf whp}}

\newcommand{\ignore}[1]{}

\def\e{\varepsilon} \def\f{\phi}   
\def\G{\Gamma}  \def\k{\kappa}
\def\z{\zeta} \def\th{\theta}    
  \def\n{\nu} \def\p{\pi}\def\x{\xi}
\def\r{\rho}  \def\s{\sigma} 
\def\t{\tau} \def\om{\omega}

\def\cY{{\cal Y}}

\def\E{{\bf E}}

\def\Pr{\mbox{{\bf Pr}}}
\def\whp{{\bf whp}}

\def\bX{\bar{X}}
\def\bX{Y}

\newcommand{\scr}[2]{_{#1}^{(#2)}}

\def\kappaNew{r}
\def\k{\kappaNew}
\def\vH{\vec{H}}
\def\ve{\vec{e}}
\title{On rainbow Hamilton cycles in random hypergraphs}

\author{Andrzej Dudek\thanks{Supported in part by Simons Foundation grant 522400.} \qquad Sean English\\
\small Department of Mathematics\\[-0.8ex]
\small Western Michigan University\\[-0.8ex] 
\small Kalamazoo, U.S.A.\\
\small\tt \{andrzej.dudek, sean.j.english\}@wmich.edu\\
\and
Alan Frieze\thanks{Supported in part by NSF grant DMS1661063.}\\
\small Department of Mathematical Sciences\\[-0.8ex]
\small Carnegie Mellon University\\[-0.8ex] 
\small Pittsburgh, U.S.A.\\
\small\tt alan@random.math.cmu.edu}

\begin{document}

\maketitle

\begin{abstract}
 Let $H_{n,p,\kappaNew}^{(k)}$ denote a \emph{randomly colored random hypergraph}, constructed on the vertex set $[n]$ by taking each $k$-tuple independently with probability $p$, and then independently coloring it with a random color from the set $[\kappaNew]$. Let $H$ be a $k$-uniform hypergraph of order $n$. An \emph{$\ell$-Hamilton cycle} is a spanning subhypergraph $C$ of  $H$  with $n/(k-\ell)$ edges and such that for some cyclic ordering of the vertices each edge of $C$ consists of $k$ consecutive vertices and every pair of adjacent edges in $C$ intersects in precisely $\ell$ vertices. 

In this note we study the existence of \emph{rainbow} $\ell$-Hamilton cycles (that is every edge receives a different color) in $H_{n,p,\kappaNew}^{(k)}$. We mainly focus on the most restrictive case when $\kappaNew = n/(k-\ell)$. In particular, we show that for the so called tight Hamilton cycles ($\ell=k-1$) $p = e^2/n$ is the sharp threshold for the existence of a rainbow tight Hamilton cycle in $H_{n,p,n}^{(k)}$ for each $k\ge 4$.
\end{abstract}

\maketitle

\section{Introduction}

Suppose that $k>\ell\ge 1$. An {\em $\ell$-Hamilton cycle} $C$
in a $k$-uniform hypergraph $H=(V,\cE)$ on $n$ vertices is a
collection of $m_\ell=n/(k-\ell)$ edges of $H$ such that for some cyclic order
of $[n]$ every edge consists of $k$ consecutive vertices and for
every pair of consecutive edges $E_{i-1},E_i$ in $C$ (in the natural
ordering of the edges) we have $|E_{i-1}\cap E_i|=\ell$ (see Figure~\ref{fig:cycles}). Thus, in every
$\ell$-Hamilton cycle the sets $C_i=E_i\setminus
E_{i-1},\,i=1,2,\ldots,m_\ell$, are a partition of $V$ into sets of
size $k-\ell$. Hence, $m_{\ell}=n/(k-\ell)$. We thus always assume, when discussing $\ell$-Hamilton cycles, that this necessary condition, $k-\ell$ divides $n$,  is fulfilled. In the
literature, when $\ell=k-1$ we have a {\em tight} Hamilton cycle and
when $\ell=1$ we have a {\em loose} Hamilton cycle.

\begin{figure}[h]
\centering
\includegraphics[scale=0.8]{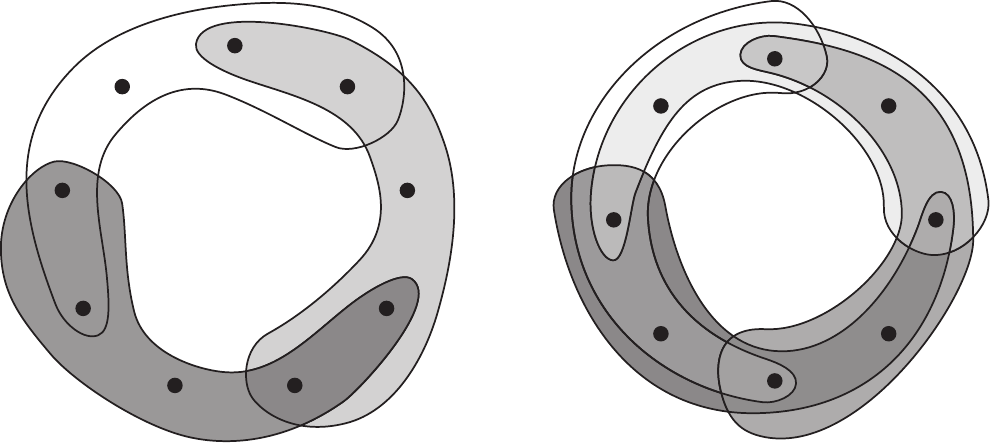}
\caption{A 2-Hamilton and a 3-Hamilton 5-uniform cycles.}
\label{fig:cycles}
\end{figure}

Let $H_{n,p}^{(k)}$ denote a \emph{random hypergraph}, constructed on the vertex set $[n]$ by taking each $k$-tuple from $\binom{[n]}{k}$ independently with probability $p$. When $k=2$ we have the well-known Erd\H{o}s-R\'enyi-Gilbert model $G_{n,p}$.

The threshold for the existence of Hamilton cycles in the random graph $G_{n,p}$ has been known
for many years, see, e.g., \cite{AKS}, \cite{Boll} and \cite{KS}.
Recently these results were extended to hypergraphs, see, e.g., \cite{ABKP, DF, DF_loose, DFLS, Fer, F, FK, G, NS, PP}. Below we summarize some of them.

In the following and throughout the paper, $\omega = \omega(n)$ can be any function tending to infinity with~$n$. All logarithms in this paper are natural (base $e$). Recall that an event $\cE_n$ occurs {\em with high probability}, or \whp\ for brevity, if $\lim_{n\rightarrow\infty}\Pr(\cE_n)=1$. 

\begin{theorem}[\cite{DF}]\label{thm:tight}
Let $\e>0$ be fixed. Then:

\begin{enumerate}[(i)]
\item\label{thm:tight:1} For all integers $k> \ell\geq 2$, if $p\leq  (1-\e)e^{k-\ell}/n^{k-\ell}$, then \whp\ $H_{n,p}^{(k)}$ is not $\ell$-Hamiltonian.
\item\label{thm:tight:2} For all integers $k>\ell \ge 3$, there exists a
constant $K=K(k)$ such that if $p\geq K/n^{k-\ell}$
and $n$ is a multiple of $k-\ell$ then
 $H_{n,p}^{(k)}$ is $\ell$-Hamiltonian \whp. \label{thm:a}
\item If $k>\ell=2$ and $p\geq \omega/n^{k-2}$
and $n$ is a multiple of $k-2$,
then $H_{n,p}^{(k)}$ is $2$-Hamiltonian \whp.\label{thm:b}
\item\label{thm:tight:4} For all $k\geq 4$, if $p\geq  (1+\e)e/n$, then
\whp\ $H_{n,p}^{(k)}$ is $(k-1)$-Hamiltonian, i.e. it contains a \emph{tight} Hamilton cycle.
\end{enumerate}
\end{theorem}
\noindent
In particular, this theorem shows that $e/n$ is the sharp threshold for the existence of a tight
Hamilton cycle in a $k$-uniform hypergraph, when $k\ge 4$. As it was explained in~\cite{DF}, quite 
surprisingly, the proof of \eqref{thm:tight:2}-\eqref{thm:tight:4} in Theorem~\ref{thm:tight} is based on the second moment method.
\begin{theorem}[\cite{DF_loose, DFLS, Fer, F}]\label{thm:loose}
Fix $k\ge 3$ and suppose that $n$ is a multiple of $k-1$. Let $p\ge \omega (\log n)/n^{k-1}$. Then, \whp\  $H_{n,p}^{(k)}$  contains a loose Hamilton cycle.
\end{theorem}
\noindent
Thus, $(\log n)/ n^{k-1}$ is the asymptotic threshold for the
existence of loose Hamilton cycles. This is because
if $p\leq (1-\e) (k-1)! (\log n) / n^{k-1}$ and $\e>0$ is
constant, then \whp\ $H_{n,p}^{(k)}$ contains isolated vertices. 

In this note we study the existence of rainbow Hamilton cycles in $H_{n,p}^{(k)}$ with independently colored edges. Let $H_{n,p,\kappaNew}^{(k)}$ denote a \emph{randomly colored random hypergraph}, constructed on the vertex set $[n]$ by taking each $k$-tuple independently with probability $p$, and then independently coloring it with a random color from the set $[\kappaNew]$. We also denote $H_{n,p,\kappaNew}^{(2)}$ by $G_{n,p,\kappaNew}$. Rainbow properties of $G_{n,p,\kappaNew}$ attracted a considerable amount of attention, see, e.g., \cite{BF, CF, FerK, FL, FNP}. 

Here we only focus on rainbow Hamilton cycles, which are Hamilton cycles where every edge of the cycle receives a different color. Improving the previous results of Cooper and Frieze~\cite{CF} and Frieze and Loh~\cite{FL}, Ferber and Krivelevich~\cite{FerK} determined the very sharp threshold for the existence of the rainbow Hamilton cycle in $G_{n,p,\kappaNew}$ assuming nearly optimal number of colors.
\begin{theorem}[\cite{FerK}]
Let $\varepsilon>0$, $\kappaNew = (1 + \varepsilon)n$ and let $p = (\log n+\log \log n+\omega)/n$. Then, \whp\ $G_{n,p,\kappaNew}$ contains a rainbow Hamilton cycle. 
\end{theorem}
\noindent
For expressions such as $\kappaNew = (1+\varepsilon)n$ that clearly have to be an integer, we round up or down but do not specify which: the reader can choose either one, without affecting the argument.

Ferber and Krivelevich~\cite{FerK} were the first to study rainbow Hamilton cycles in $H_{n,p}^{(k)}$. They showed the following. (Recall that $m_\ell = n/(k-\ell)$ is the number of edges in an $\ell$-Hamilton cycle.)
\begin{theorem}[\cite{FerK}]\label{thm:fk}
Let $k>\ell \ge 1$ be integers. Suppose that $n$ is a multiple of $k-\ell$. Let $p\in [0,1]$ be such that \whp\  $H_{n,p}^{(k)}$ contains an $\ell$-Hamilton cycle.
Then, for every $\varepsilon = \varepsilon(n)\ge 0$, letting $\kappaNew = (1+\varepsilon) m_\ell$ and $q = \kappaNew p / (\varepsilon m_\ell + 1)$ we have that \whp\ 
$H_{n,q,\kappaNew}^{(k)}$  contains a rainbow $\ell$-Hamilton cycle.
\end{theorem}
\noindent
Observe that if $\varepsilon$ is a constant, then by losing a multiplicative constant in the threshold, a rainbow $\ell$-Hamilton \whp\ exists. By combining this result with Theorems \ref{thm:tight} and \ref{thm:loose} one can obtain some explicit values of $q$. However, for small $\varepsilon$ (including $\varepsilon=0$) Theorem~\ref{thm:fk} does not provide optimal $q$. In our results we focus on the case when $\kappaNew = m_\ell$. (But we also allow more colors.) Here we state our first result.

\begin{theorem}\label{thm1}
Let $k > \ell \ge 2$ and $\e>0$ be fixed.
Let $c\geq 1/(k-\ell)$ and $\kappaNew=cn$. Then:

\begin{enumerate}[(i)]
\item\label{thm1:1} For all integers $k> \ell\geq 2$, if
\[
p\leq
\begin{cases}
 (1-\e)e^{k-\ell+1}/n^{k-\ell} & \text{ if } c = 1/(k-\ell)\\
 (1-\e)\bfrac{c-1/(k-\ell)}{c}^{(k-\ell)c-1}e^{k-\ell+1}/n^{k-\ell} & \text{ if } c > 1/(k-\ell),
\end{cases}
\] then
\whp\ $H_{n,p,\kappaNew}^{(k)}$ is not rainbow $\ell$-Hamiltonian.
\item\label{thm1:2} For all integers $k>\ell \ge 3$, there exists a
constant $K=K(k)$ such that if $p\geq K/n^{k-\ell}$
and $n$ is a multiple of $k-\ell$ then
 $H_{n,p,\kappaNew}^{(k)}$ is rainbow $\ell$-Hamiltonian \whp. \label{thm:a}
\item\label{thm1:3} If $k>\ell=2$ and $p\geq \omega/n^{k-2}$
and $n$ is a multiple of $k-2$,
then $H_{n,p,\kappaNew}^{(k)}$ is rainbow $2$-Hamiltonian \whp.\label{thm:b}
\item\label{thm1:4} For all $k\geq 4$, if
\[
p\geq 
\begin{cases}
(1+\e)e^2/n & \text{ if } c=1\\
(1+\e) \bfrac{c-1}c^{c-1} e^2 / n & \text{ if } c>1,
\end{cases}
\] then
\whp\ $H_{n,p,\kappaNew}^{(k)}$ is rainbow $(k-1)$-Hamiltonian, i.e. it contains a rainbow \emph{tight} Hamilton cycle.
\label{thm:c}
\end{enumerate}
\end{theorem}
\noindent
Consequently, if $k\ge 4$, then
\[
p = 
\begin{cases}
e^2/n & \text{ if } c=1\\
\bfrac{c-1}c^{c-1} e^2 / n & \text{ if } c>1
\end{cases}
\] 
is the sharp threshold for the existence of a rainbow tight Hamilton cycle. Furthermore, observe 
that $\lim_{c\to 1^+} \bfrac{c-1}{c}^{c-1} = 1$. Thus, in \eqref{thm1:4} the case $c>1$ approaches  the case $c=1$ in the continuous way. Finally, also observe that $\lim_{c\to \infty} \bfrac{c-1}{c}^{c-1} = 1/e$. Hence, when $c$ tends to infinity (that means that each edge receives a different color) the threshold function is $e/n$, which is consistent with Theorem~\ref{thm:tight}. The proof of Theorem~\ref{thm1} modifies the proof of Theorem~\ref{thm:tight}.

We also establish a similar result for loose Hamilton cycles. Recall that a loose Hamilton cycle of order $n$ has exactly $n/(k-1)$ edges. So for a rainbow loose Hamilton cycle we always need at least $n/(k-1)$ colors. Here we only consider this most restrictive case with $\kappaNew = n/(k-1)$.
\begin{theorem}\label{thm2}
Fix $k\ge 3$ and suppose that $n$ is a multiple of $k-1$. Let $\kappaNew=n/(k-1)$ and $p\ge \omega (\log n)/n^{k-1} $. Then, \whp\ $H_{n,p,\kappaNew}^{(k)}$ contains a rainbow loose Hamilton cycle.
\end{theorem}
\noindent
The proof is a modification of the proof of Theorem~\ref{thm:loose}.

\medskip

{\bf Some notation:} For sequences $A_n,B_n,n\geq 1$ we write $A_n\approx B_n$ to mean that $A_n=(1+o(1))B_n$ as $n\to\infty$. Similarly, we write $A_n\lesssim B_n$ to mean that $A_n\leq (1+o(1))B_n$ as $n\to \infty$.

\section{Proof of Theorem~\ref{thm1}}

The proof modifies the proof of Theorem~3 from~\cite{DF}. 

Let $([n],\cE)$ be a $k$-uniform hypergraph.
A permutation $\p$ of $[n]$ is {\em $\ell$-Hamilton cycle inducing} if
$$E_\p(i)=\set{\p((i-1)(k-\ell)+j):\;j\in [k]}\in \cE\ for\ all\ i\in [n/(k-\ell)].$$
(We use the convention $\p(n+r)=\p(r)$ for $r>0$.)
Let the term {\em hamperm} refer to such a permutation.

Let $Y$ be the random variable that counts the number of rainbow
hamperms $\p$ for $H_{n,p,\kappaNew}^{(k)}$. Every $\ell$-Hamilton cycle induces at least one hamperm and so we can concentrate on
estimating $\Pr(Y > 0)$.

Observe that 
\[
\E(Y) = n! \cdot p^{n/(k-\ell)} \cdot \frac{(\k)_{n/(k-\ell)}}{\k^{n/(k-\ell)}},
\]
where $(x)_t = x(x-1)\cdots(x-t+1)$ is the \emph{falling factorial}.
This is because $\pi$ induces an $\ell$-Hamilton cycle if and only if a certain $n/(k-\ell)$ edges are present and are colored rainbow. 

Now let $c > 1/(k-\ell)$. Then, by Stirling's formula we get
\begin{align*}
\E(Y) 
&= n!p^{n/(k-\ell)} \frac{\k!}{\k^{n/(k-\ell)} (\k-n/(k-\ell))! }\\
&\approx\sqrt{2\pi n} \bfrac{n}{e}^n p^{n/(k-\ell)}    \frac{\sqrt{\frac{\kappaNew}{\kappaNew - n/(k-\ell)}}\bfrac{\k}{e}^\k}{ \k^{n/(k-\ell)} \bfrac{\k-n/(k-\ell)}{e}^{\k-n/(k-\ell)}}\\
&= \sqrt{\frac{2\pi n \kappaNew}{\kappaNew - n/(k-\ell)}} \left( \frac{n p^{1/(k-\ell)}}{e^{1 + 1/(k-\ell)}} \cdot \bfrac{\k}{\k-n/(k-\ell)}^{\k/n-1/(k-\ell)} \right)^n\\ 
&= \sqrt{\frac{2\pi n\kappaNew}{\kappaNew - n/(k-\ell)}} \left( \frac{n p^{1/(k-\ell)}}{e^{1 + 1/(k-\ell)}} \cdot \bfrac{c}{c-1/(k-\ell)}^{c-1/(k-\ell)} \right)^n. 
\end{align*}
Similarly for $c= 1/(k-\ell)$ we get
\[
\E(Y) \approx 2\pi n \sqrt{\frac{1}{k-\ell}} \left( \frac{n p^{1/(k-\ell)}}{e^{1 + 1/(k-\ell)}}  \right)^n. 
\]
Thus, if \[
p\leq
\begin{cases}
 (1-\e)e^{k-\ell+1}/n^{k-\ell} & \text{ if } c = 1/(k-\ell)\\
 (1-\e)\bfrac{c-1/(k-\ell)}{c}^{(k-\ell)c-1}e^{k-\ell+1}/n^{k-\ell} & \text{ if } c > 1/(k-\ell),
\end{cases}
\] then
$\E(Y) = o(1)$. This verifies part~\eqref{thm1:1}.

Now we prove parts \eqref{thm1:2}-\eqref{thm1:4} by the second moment method. First observe that if
\[
p\geq
\begin{cases}
 (1+\e)e^{k-\ell+1}/n^{k-\ell} & \text{ if } c = 1/(k-\ell)\\
 (1+\e)\bfrac{c-1/(k-\ell)}{c}^{(k-\ell)c-1}e^{k-\ell+1}/n^{k-\ell} & \text{ if } c > 1/(k-\ell),
\end{cases}
\] then
$\E(Y)\to\infty$ together with $n$.

Fix a hamperm $\p$. Let $H(\p)=(E_\p(1),E_\p(2),\ldots,E_{\p}(m_\ell))$ be the Hamilton cycle induced by
 $\p$. Then let $N(b,a)$ be the number of permutations $\p'$
such that $|E(H(\p))\cap E(H(\p^\prime))|=b$ and $E(H(\p))\cap E(H(\p^\prime))$
consists of $a$ edge disjoint paths.
Here a path is a maximal sub-sequence $F_1,F_2,\ldots,F_q$ of the edges of $H(\p)$ such that $F_i\cap F_{i+1}\neq \emptyset$
for $1\leq i<q$. The set $\bigcup_{j=1}^qF_j$ may contain other edges of $H(\p)$.
Observe that $N(b,a)$ does not depend on $\p$.

Now,
\begin{multline*}
\frac{\E(Y^2)}{\E(Y)^2}
\leq \frac{n!N(0,0)p^{2n/(k-\ell)} \left(\frac{(\k)_{n/(k-\ell)}}{\k^{n/(k-\ell)}}\right)^2}{\E(Y)^2}\\+\sum_{b=1}^{n/(k-\ell)} \sum_{a=1}^{b}
\frac{n! N(b,a) p^{2n/(k-\ell)-b}} {\E(Y)^2}\cdot  \frac{(\k)_{n/(k-\ell)}}{\k^{n/(k-\ell)}} \cdot \frac{(\k-b)_{n/(k-\ell)-b}}{\k^{n/(k-\ell)-b}}.
\end{multline*}
Since trivially $N(0,0)\le n!$, we get
\[
\frac{\E(Y^2)}{\E(Y)^2}
\leq 1+\sum_{b=1}^{n/(k-\ell)} \sum_{a=1}^{b}
\frac{n! N(b,a) p^{2n/(k-\ell)-b}} {\E(Y)^2}\cdot  \frac{(\k)_{n/(k-\ell)}}{\k^{n/(k-\ell)}} \cdot \frac{(\k-b)_{n/(k-\ell)-b}}{\k^{n/(k-\ell)-b}}.
\]
Let $X$ be the number of $\ell$-hamperms in $H_{n,p}^{(k)}$. Then,
\[
\E(X) = n! p^{n/(k-\ell)} \quad \text{ and } \quad   \E(Y)=\E(X)\cdot \frac{(\k)_{n/(k-\ell)}}{\k^{n/(k-\ell)}}.
\]
Consequently,
\begin{align}
\frac{\E(Y^2)}{\E(Y)^2}
&\leq 1+\sum_{b=1}^{n/(k-\ell)} \sum_{a=1}^{b}
\frac{n! N(b,a) p^{2n/(k-\ell)-b}} {\E(X)^2}\cdot  \frac{(\k)_{n/(k-\ell)}}{\k^{n/(k-\ell)}} \cdot \frac{(\k-b)_{n/(k-\ell)-b}}{\k^{n/(k-\ell)-b}}  \cdot 
\bfrac{\k^{n/(k-\ell)}}{(\k)_{n/(k-\ell)}}^2 \notag \\
&=1+\sum_{b=1}^{n/(k-\ell)} \sum_{a=1}^{b} \frac{N(b,a)p^{n/(k-\ell)-b}} {\E(X)}\cdot \k^{b}\cdot \frac{(\k-b)_{n/(k-\ell)-b}}{(\k)_{n/(k-\ell)}}\notag \\
&=1+\sum_{b=1}^{n/(k-\ell)} \sum_{a=1}^{b} \frac{N(b,a)p^{n/(k-\ell)-b}} {\E(X)}\cdot \k^{b}\cdot \frac{(\k-b)!}{\k!} \notag \\
&\le 1+\sum_{b=1}^{n/(k-\ell)} \sum_{a=1}^{b} \frac{N(b,a)p^{n/(k-\ell)-b}} {\E(X)}\cdot e^{b}\bfrac{\k-b}{\k}^{\k-b}.\label{thm1:var}
\end{align}

\medskip

\textbf{Part \eqref{thm1:2}: $\ell \ge 3$}

\medskip

We trivially bound $\bfrac{\k-b}{\k}^{\k-b} \le 1$. It was shown in \cite{DF} (equation (10)) that
\[
\frac{N(b,a)p^{n/(k-\ell)-b}} {\E(X)} \lesssim \left( \frac{2k!ke^k}{n^{k-\ell}p} \right)^b \frac{1}{n^{a(\ell-2)}}.
\]
Thus, 
\begin{align*}
\frac{\E(Y^2)}{\E(Y)^2} 
&\le 1+\sum_{b=1}^{n/(k-\ell)} \sum_{a=1}^{b} \frac{N(b,a)p^{n/(k-\ell)-b}} {\E(X)}\cdot e^{b}\\
&\lesssim 1+\sum_{b=1}^{n/(k-\ell)} \sum_{a=1}^{b} \left( \frac{2k!ke^k}{n^{k-\ell}p} \right)^b \frac{1}{n^{a(\ell-2)}} \cdot e^{b}\\
&\le 1+\sum_{b=1}^{n/(k-\ell)} \sum_{a=1}^{b} \left( \frac{2k!ke^{k+1}}{n^{k-\ell}p} \right)^b \frac{1}{n^{a}}.
\end{align*}
Set $K = 4 k!ke^{k+1}$ and $p = K/n^{k-\ell}$. Thus,
\[
\frac{\E(Y^2)}{\E(Y)^2}  \le 1+\sum_{b=1}^{n/(k-\ell)} \sum_{a=1}^{b} \frac{1}{2^b} \cdot \frac{1}{n^{a}}
\le 1+ \left( \sum_{b=1}^{n} \frac{1}{2^b}\right)
\left(\sum_{a=1}^n \frac{1}{n^{a}}\right) \approx 1.
\]

\medskip

\textbf{Part \eqref{thm1:3}: $\ell = 2$}

\medskip

Let $p \ge \omega/n^{k-2}$. Similarly as in the previous case
\begin{align*}
\frac{\E(Y^2)}{\E(Y)^2} 
&\le 1+\sum_{b=1}^{n/(k-2)} \sum_{a=1}^{b} \left( \frac{2k!ke^k}{n^{k-2}p} \right)^b  \cdot e^{b}\\
&\le 1+\sum_{b=1}^{n/(k-2)} \sum_{a=1}^{b} \left( \frac{2k!ke^{k+1}}{\omega} \right)^b\\
&\le 1+\sum_{b=1}^{n/(k-2)} b \left( \frac{2k!ke^{k+1}}{\omega} \right)^b \approx 1.
\end{align*}

\medskip

\textbf{Part \eqref{thm1:4}: $\ell = k-1$} (tight cycles)

\medskip

If $c=1$ (that means $\k = n$), then we trivially bound $\bfrac{\k-b}{\k}^{\k-b} \le 1$. Otherwise, we  use a simple fact.
\begin{claim} Let $\kappaNew = cn$, where $c >1$. Then,
\[
\max_{0<b\leq n}{\bfrac{\kappaNew-b}{\kappaNew}^\frac{\kappaNew-b}{b}}=\bfrac{\kappaNew-n}{\kappaNew}^\frac{\kappaNew-n}{n}=\bfrac{c-1}c^{c-1}.
\]
\end{claim}
\begin{proof}[Proof of the claim]
Let $x=b/n$. Note that since $1\leq b\leq n$,  $x\in(0,1]$ and $c>x$. Then
\[
\bfrac{\kappaNew-b}{\kappaNew}^\frac{\kappaNew-b}{b}=\bfrac{c-b/n}{c}^\frac{c-b/n}{b/n}=\bfrac{c-x}{c}^\frac{c-x}{x}.
\]
Taking the derivative gives us
\[
\frac{d}{dx}\bfrac{c-x}{c}^\frac{c-x}{x}=-\frac{c}{x^2}\bfrac{c-x}{c}^\frac{c-x}{x}\left(\log\bfrac{c-x}c+\frac{x}c\right).
\]
Since $\frac{c}{x^2}\bfrac{c-x}{c}^\frac{c-x}{x}>0$, we have
\[
\sign\left(\frac{d}{dx}\bfrac{c-x}{c}^\frac{c-x}{x}\right)=-\sign\left(\log\bfrac{c-x}c+\frac{x}c\right)=-\sign\left(\log\left(1-\frac{x}c\right)+\frac{x}c\right)
\]
and since $\log\left(1-\frac{x}c\right) < \log e^{-\frac{x}c} = -\frac{x}c$ we get $\log\left(1-\frac{x}c\right)+\frac{x}c< 0$. Thus
\[
\frac{d}{dx}\bfrac{c-x}{c}^\frac{c-x}{x}>0
\]
for $0<x\leq 1$ and $c>x$. Thus $\bfrac{c-x}{c}^\frac{c-x}{x}$ is maximized at $x=1$ in our domain, which corresponds to $b=n$, proving the claim.
\end{proof}

Due to~\eqref{thm1:var} and the above claim we obtain
\[
\frac{\E(Y^2)}{\E(Y)^2} \leq
\begin{cases}
1+\sum_{b=1}^{n} \sum_{a=1}^{b} \frac{N(b,a)p^{n-b}} {\E(X)}\cdot e^b, & \text{ if } c=1\\
1+\sum_{b=1}^{n} \sum_{a=1}^{b} \frac{N(b,a)p^{n-b}} {\E(X)}\cdot \left( e \bfrac{c-1}c^{c-1} \right)^b, & \text{ if } c>1.\\
\end{cases}
\]
Moreover, it was shown in~\cite{DF} (equation (13)) that for $k\ge 4$,
\[
\sum_{b=1}^{n} \sum_{a=1}^{b} \frac{N(b,a)p^{n-b}} {\E(X)}
\le \frac{2c_kk!e^{k-1}}{n^{k-3}}\exp\set{\frac{2k!e^{k-1}}{n^{k-4}}}
\sum_{b=1}^{n}\bfrac{e}{np}^b
\]
for some positive constant $c_k$ that depends on $k$ only. Thus, 
\[
\frac{\E(Y^2)}{\E(Y)^2} \leq
\begin{cases}
1+\frac{2c_kk!e^{k-1}}{n^{k-3}}\exp\set{\frac{2k!e^{k-1}}{n^{k-4}}}
\sum_{b=1}^{n}\bfrac{e}{np}^b \cdot e^b, & \text{ if } c=1\\
1+\frac{2c_kk!e^{k-1}}{n^{k-3}}\exp\set{\frac{2k!e^{k-1}}{n^{k-4}}}
\sum_{b=1}^{n}\bfrac{e}{np}^b \cdot \left( e \bfrac{c-1}c^{c-1} \right)^b, & \text{ if } c>1.\\
\end{cases}
\]
Hence, both for $c=1$, $p\geq \frac{(1+\varepsilon)e^2}{n}$ and for $c>1$, $p\geq (1+\varepsilon) \bfrac{c-1}c^{c-1}\frac{e^2}{n}$, we get that
\[
\frac{\E(Y^2)}{\E(Y)^2} \le 1+\frac{2c_kk!e^{k-1}}{n^{k-3}}\exp\set{\frac{2k!e^{k-1}}{n^{k-4}}}
\sum_{b=1}^{n} \frac{1}{(1+\varepsilon)^b} \approx 1.
\] 

\medskip

In all three cases we showed that $\frac{\E(Y^2)}{\E(Y)^2} \lesssim 1$. Thus, the Chebyshev inequality completes the proof of Theorem~\ref{thm1}.

\section{Proof of Theorem~\ref{thm2}}
Let $n = (k-1)m$ and assume that $m$ is even. Clearly, $m=m_1 = \kappaNew$. In this case the proof is a straightforward modification of the proof of Theorem~2 from~\cite{DF_loose}. The idea being that we interpret an edge $\set{x_1,x_2,\ldots,x_k}$ of color $c\in [r]$ as an edge $\set{x_1,x_2,\ldots,x_k,c}$ in an auxilliary $(k+1)$-uniform hypergraph. Care must be taken in the proofs that (i) the components corresponding to colors are not used as the intersections of edges of the cycle and (ii) we do not have edges $\set{x_1,x_2,\ldots,x_k,c_1}$ and $\set{x_1,x_2,\ldots,x_k,c_2}$ i.e. we give the same edge two colors. Neither of these requirements are difficult to ensure. Indeed, requirement (ii) happens \whp.

In a little more detail, let $X=[m]$ and $Y=[m+1,n]$ and $Z=[n+1,n+m]$. Given $H=H_{n,p,m}^{(k)}$ we define the $(k+1)$-uniform hypergraph $\G$ with vertex set $[n]$ and an edge $\f(e)$ for each edge $e=\set{x_1,x_2,y_1,\ldots,y_{k-2}}$ of $H$ that satisfies $|e\cap X|=2$. Here $x_1,x_2\in X$ and $y_i\in Y,1\leq i\leq k-2$. We then let $\f(e)=\set{x_1,x_2,y_1,\ldots,y_{k-2},c(e)+n}$, where $c(e)$ is the color of $e$ and $c(e)+n \in Z$. The proof in \cite{DF_loose} can be adapted (and therefore we need to assume that $m$ is even) to show that \whp\ $\G$ contains a loose Hamilton cycle where consecutive edges intersect in vertices of $X$. We will give sufficient detail in Appendix A to justify this claim. A loose Hamilton cycle in $\G$ corresponds to a rainbow loose Hamilton cycle of $H$, where we re-interpret the vertex $z$ of an edge as the color $z-n$.

We can easily remove the requirement that $m$ be even by using an idea of Ferber \cite{Fer}.
In particular, one can follow his proof of Theorem \ref{thm:loose} to show that $\G$ contains a loose Hamilton cycle in this case. More details are given in Appendix B.

\appendix\label{}
\section{Modifying the proof in \cite{DF_loose}}
Suppose that $p=\om (\log n)/n^{k-1}$, where $\om=o(\log n)$ and $\om\to \infty$. Let $M=\binom{n}{k}p$ and consider a random $(k+1)$-uniform hypergraph $K$ with approximately $M'\approx M$ edges. Then 
\begin{align*}
\Pr(\exists\, e_1,e_2\in E(K):|e_1\cap e_2|=k)&\leq  \binom{n}{k+1} \binom{k+1}{k} n
\frac{\binom{\binom{n}{k+1}-2}{M'-2}}{\binom{\binom{n}{k+1}}{M'}}\\
&\leq n^{k+2}\bfrac{M'}{\binom{n}{k+1}}^2\\
&\leq n^{k+1}\bfrac{2(k+1)!\om n\log n}{n^{k+1}}^2\\
&=o(1).
\end{align*}
In this way we can justify viewing $H_{n,p,m}^{(k)}$  as a random $(k+1)$-uniform hypergraph. This would be a problem if the latter model gave an edge more than one color.

Let $m=2m_1$. Then $m_1$ will replace $m$ in the proof in \cite{DF_loose}. The proof in \cite{DF_loose} involves proving that \whp~$H_{n,p}^{(k)}$ contains a loose Hamilton cycle that respects a certain vertex partition. Such a Hamilton cycle will consist of $2m_1$ edges of the form $\{x_{i},x_{i+1},y_{i,1},\ldots,y_{i,\k}\}$, where $\k=k-2$, $1\le i\le 2m_1$, $x_{2m_1+1}=x_1$, $\set{x_1,\ldots,x_{2m_1}}=X$ and $\set{y_{1,1},\ldots,y_{2m_1,\k}}=Y$. This is done as follows: we choose a large positive integer $d$. Let $\cX$ be a set of size $2dm_1$ representing
$d$ copies of each $x\in X$. Denote the $j$th copy of $x\in X$ by $x\scr{}{j}\in \cX$ and let $\cX_x=\set{x\scr{}{j},\,j=1,2,\ldots,d}$. Then let $X_1,X_2,\ldots,X_{d}$ be a uniform random partition of $\cX$ into $d$ sets of size $2m_1$. Define $\psi_1:\cX\to X$ by $\psi_1(x\scr{}{j})=x$ for all $j$ and $x\in X$. Similarly, we let $\cY$ be a set of size $d\k m_1$ representing
$d/2$ copies of each $y\in \bX$. Denote the $j$th copy of $y\in \bX$ by $y\scr{}{j} \in \cY$ and let $\cY_y=\set{y\scr{}{j},\,j=1,2,\ldots,d/2}$.
Then let $Y_1,Y_2,\ldots,Y_{d}$ be a uniform random partition of $\cY$ into $d$ sets of size $\k m_1$. Define $\psi_2:\cY\to \bX$ by $\psi_2(y\scr{}{j})=y$ for all $y\in\bX$. Finally, let $\psi:\binom{\cX}{2}\times \binom{\cY}{\k} \to X^2 \times \bX^{\k}$ be such that $\psi(\n_1,\n_2,\x_1,\x_2,\dots,\x_\k) =
(\psi_1(\n_1),\psi_1(\n_2),\psi_2(\x_1),\psi_2(\x_2),\dots,\psi_2(\x_\k))$. 

All we need do is add a set $\cZ$ of size $dm_1$ representing $d/2$ copies of each $z\in Z$. We denote the $j$th copy of $z\in Z$ by $z\scr{}{j} \in \cZ$ and let $\cZ_z=\set{z\scr{}{j},\,j=1,2,\ldots,d/2}$. Then let $Z_1,Z_2,\ldots,Z_{d}$ be a uniform random partition of $\cZ$ into $d$ sets of size $m_1$. Define $\psi_3:\cZ\to Z$ by $\psi_3(z\scr{}{j})=z$ for all $z\in Z$. We then modify $\psi$ so that $\psi:\binom{\cX}{2}\times \binom{\cY}{\k} \times \cZ\to X^2 \times \bX^{\k}\times Z$ be such that $\psi(\n_1,\n_2,\x_1,\x_2,\dots,\x_\k,\z) =
(\psi_1(\n_1),\psi_1(\n_2),\psi_2(\x_1),\psi_2(\x_2),\dots,\psi_2(\x_\k),\psi_3(\z))$. After this, the proof in \cite{DF_loose} can be carried out with straightforward modifications involving adding a component for members of $Z$.

\section{Removing the requirement that $m$ is even}
We begin by defining a random colored directed hypergraph $D=D_{V,C,q}^{(k)}$. Here $V$ is the vertex set, $C$ is the set of colors for each edge $e\in \binom{V}{k}$ and for each of the $k!$ orderings $\ve$ of the vertices in $e$, we include $\ve$ as an oriented edge with probability $q$ and give it a color $c$ chosen uniformly from $C$. When $V=[n],C=[r]$, we refer to this graph as $\vH_{n,q,r}^{(k)}$. Generalizing Lemma 6 of \cite{Fer} we have the following lemma.
\begin{lemma}\label{lemcd}
If $q$ satisfies $q-2q^2=p$ then
\begin{multline}\label{comp}
\Pr(\vH_{n,q,r}^{(k)}\text{ has a loose rainbow Hamilton cycle})\geq \\
\Pr(H_{n,p,r}^{(k)}\text{ has a loose rainbow Hamilton cycle})
\end{multline}
\end{lemma}
\begin{proof}
Using an idea of McDiarmid \cite{McD} we define a sequence of random colored directed hypergraphs $\G_i,i=0,1,\ldots,N=\binom{n}{k}$. Let $e_1,e_2,\ldots,e_N$ be an enumeration of $\binom{[n]}{k}$ and let $\eta_i$ denote the $k!$ distinct orderings of the elements of $e_i$. In $\G_i$, we add all of $\eta_j,j\geq i$ to our graph with probability $p$ and none with probability $1-p$. For $j<i$ we add each member of $\eta_i$ independently to $\G_i$ with probability $q$. Thus $\G_0=H_{n,p,r}^{(k)}$ and $\G_N=\vH_{n,q,r}^{(k)}$ and to prove \eqref{comp} we show that for each $i\geq 0$,
\begin{multline}\label{comp1}
\Pr(\G_{i}\text{ has a loose rainbow Hamilton cycle})\geq \\
\Pr(\G_{i-1}\text{ has a loose rainbow Hamilton cycle}).
\end{multline}
To prove \eqref{comp1} we condition on all edges associated with $\eta_j,j\neq i$. Denote this conditioning by $\cE$. Then $\G_{i-1}$ and $\G_{i}$ differ only in the existence of the edges in $\eta_i$. We focus on the case where the existence of a loose rainbow cycle depends on the edges of $\eta_i$. (In the remaining cases, there is a cycle without these edges or there is no such cycle regardless of these edges.) For each $\s\in \eta_i$ let $C_\s$ be the set of colors such that if we include $\s$ of a color in $C_\s$ then we create a new loose rainbow cycle. Then,
\beq{comp2}
\Pr(\G_{i-1}\text{ has a loose rainbow Hamilton cycle}\mid \cE)=\frac{p\left|\bigcup_\s C_\s\right|}{|C|}\leq \frac{p\sum_\s|C_\s|}{|C|}.
\eeq
On the other hand,
\begin{multline}\label{comp3}
\Pr(\G_{i}\text{ has a loose rainbow Hamilton cycle}\mid \cE)\geq \frac{q\sum_\s|C_\s|}{|C|}-\frac{q^2\sum_{\s,\t}|C_\s|\, |C_\t|}{|C|^2}=\\ \frac{q\sum_\s|C_\s|}{|C|}-\bfrac{q\sum_\s|C_\s|}{|C|}^2.
\end{multline}
Comparing \eqref{comp2}, \eqref{comp3} we see that \eqref{comp1} holds. This follows from the fact that $q\th-q^2\th^2\geq p\th$ for $0\leq \th\leq 1$, by our choice of $q$. 
\end{proof}

Now suppose that $n=(k-1)m$ and $m$ is not even. The next idea is to generate $H_{n,p,r}^{(k)}$ as the union of independent random hypergraphs $\bigcup_{i=0}^{\om_1} H_i$, $\om_1=\om^{1/2}$. Let the hypergraph $H_0=H_{n,p/2,r}^{(k)}$. For $i\geq 1$, we let the $H_i$ be independent copies of $\vH_{n,q,r}$, where we ignore orientation. Here $q$ satisfies $1-p=(1-p/2)(1-q)^{k!\om}$ and so the decomposition is valid. Next note that the probability an edge occurs twice as an edge of an $H_i$ is $O(\om n^kp^2)=o(1)$. So \whp\ there is no problem with an edge having two colors.

Now fix an edge $e^*$ of $H_0$ and let $c$ be its color in $H_0$. Also fix an ordering $x_1,x_2,\ldots,x_k$ of the vertices of $e^*$. Supplying $e^*$ is the only role of $H_0$. As we expose $H_i$, we construct a random colored directed hypergraph $D_i$.  The vertices of $D_i$ are $V^*\setminus \set{x_1,x_2,\ldots,x_k})\cup v^*)$ and $C^*=[r]\setminus \set{c}$, where $v^*$ is a new vertex. Note that $|V^*|=(k-1)(m-1)$. An arc $e$ gives rise to an edge of $H_i$ if it satisfies one of the following:
\begin{enumerate}[(a)]
\item $e\cap e^*=\emptyset$.
\item $e\cap e^*=\set{x_1}$ and $x_1$ is not the first vertex of $e$. In which case we add the edge $(e\setminus \set{x_1})\cup v^*$ to $D_i$.
\item $e\cap e^*=\set{x_k}$ and $x_k$ is the first vertex of $e$. In which case we add the edge $(e\setminus \set{x_k})\cup v^*$ to $D_i$.
\end{enumerate}
We then observe that by this construction, each $D_i$ is distributed as $D_{V^*,C^*,q}$. It follows from Lemma \ref{lemcd} that if $\r=q-2q^2$, then
\begin{multline}\label{f1}
\Pr(D_i\text{ contains a loose rainbow Hamilton cycle})\geq\\
\Pr(H_{n-(k-1),\r,r-1}\text{ contains a loose rainbow Hamilton cycle}).
\end{multline}
We have $\r=\Omega(p/\om_1)=\Omega(\om^{1/2}n^{-(k-1)}\log n$, so $D_i$ contains  a loose rainbow Hamilton cycle \whp\ from the case where $m$ is even. Now, by symmetry $v^*$ is a start/end point of the edge of such a cycle with probability $2/k$. If $D_i$ contains such a cycle and $v^*$ is a start/end, then when it is replaced by $e^*$ in the implied permutation of $V^*$, we obtain a loose rainbow Hamilton cycle in $H_i$. Thus, the probability that $H_{n,p,r}^{(k)}$ contains no loose rainbow Hamilton cycle, given $e^*$ exists, can be bounded by $(1-2/k-o(1))^{\om_1}=o(1)$.
\end{document}